\documentclass[a4paper,11pt]{amsart}

\usepackage{style}

\title[Symplectic birational transformations on O'Grady's sixfolds]{
Symplectic birational transformations of finite order on O'Grady's sixfolds
}

\author{Annalisa Grossi}
\address[Annalisa Grossi]{Lehrstuhl für Algebra und Zahlentheorie, Institut für Mathematik, Universität Augsburg, Universitätsstr. 14, 86135 Augsburg, Germany}
\curraddr{Fakultät für Mathematik, Technische Universität Chemnitz, Reichenhainer Str. 39, 09126 Chemnitz, Germany}
\email{annalisa.grossi@math.tu-chemnitz.de}

\author{Claudio Onorati}
\address[Claudio Onorati]{Matematisk institutt, Universitetet i Oslo, postboks 1053, Blindern, 0316 Oslo, Norway}
\curraddr{Dipartimento di Matematica, Universit\`a degli studi di Roma Tor Vergata, via della Ricerca Scientifica 1, 00133 Roma, Italy}
\email{onorati@mat.uniroma2.it}

\author{Davide Cesare Veniani}
\address[Davide Cesare Veniani]{Institut für Diskrete Strukturen und Symbolisches Rechnen, Universität Stuttgart, Pfaffenwaldring 57, 70569 Stuttgart, Germany}
\email{davide.veniani@mathematik.uni-stuttgart.de}

\newcommand{\annalisa}[1]{{\stepcounter{notecount} {Criptsize \color{purple}  $^{[\arabic{notecount}]}$#1 \newline}}}

\newcommand{\claudio}[1]{{\stepcounter{notecount} {Criptsize \color{teal} $^{[\arabic{notecount}]}$#1 \newline}}}

\date{\today}

\makeatletter
\@namedef{subjclassname@2020}{%
 \textup{2020} Mathematics Subject Classification}
\makeatother



\thanks{The second author gratefully acknowledges the Research Council of Norway project 250104 for financial support.}

\usepackage{framed}

\begin{document}

\maketitle

\begin{abstract}
We prove that any symplectic automorphism of finite order on a manifold of type \(\OG6\) acts trivially on the Beauville--Bogomolov--Fujiki lattice and that any birational transformation of finite order acts trivially on its discriminant group. Moreover, we classify all possible invariant and coinvariant sublattices.
\end{abstract}


\section{Introduction}

By a result of Rapagnetta~\cite{Rapagnetta:topological.invariants.OG6}, the Beauville--Bogomolov--Fujiki lattice \(\HH^2(X,\IZ)\) of an irreducible holomorphic symplectic manifold \(X\) belonging to the sporadic deformation type of dimension \(6\) constructed by O'Grady~\cite{OGrady:new.six.dimensional} is isomorphic to
\[
    \bL \coloneqq 3\bU \oplus 2[-2].
\]

In this paper we answer the question of which isometries on \(\bL\) can be induced by a symplectic birational transformation of finite order on a manifold of type~\(\OG6\).

Our first theorem is the classification of all isometries on \(\bL\) induced by symplectic \emph{automorphisms} of finite order. Quite surprisingly, this turns out to be trivial. 

\begin{theorem}[\autoref{proof:thm:aut.trivial.action}] \label{thm:aut.trivial.action}
Any symplectic automorphism of finite order on a manifold of type~\(\OG6\) acts trivially on the Beauville--Bogomolov--Fujiki lattice~\(\bL\).
\end{theorem}

This phenomenon has never been observed before in other deformation types of irreducible holomorphic symplectic manifolds. 
In fact, there exist K3 surfaces, manifolds of type \(\K3^{[n]}\) and of type \(\Kum^{[n]}\) which admit symplectic automorphisms of finite order with nontrivial action (see for instance \cite{Mongardi:towards.classification.symp.K3^n, Mongardi.Tari.Wandel:kummer.fourfolds,Mukai:finite.groups.Mathieu.K3}).

\autoref{thm:aut.trivial.action} can be rephrased by saying that any symplectic automorphism of finite order is contained in the kernel of the representation map \(\eta_* \colon \Bir(X) \rightarrow \OO(\bL)\), where \(\eta \colon \HH^2(X,\IZ) \rightarrow \bL\) is a marking (see \autoref{subsec:OG6}).
By \cite[Theorem~2.1]{Hassett.Tschinkel:Hodge.theory.lagr.planes.Kum^2}, the kernel is a deformation invariant, and in the case of manifolds of type \(\OG6\) it is isomorphic to \((\IZ/2\IZ)^8\) (see \cite{Mongardi.Wandel:OG.trivial.action}).

Our second theorem considers the induced action on the discriminant group \(\bL^\sharp \coloneqq \bL^\vee/\bL\).

\begin{theorem}[\autoref{proof:thm:bir.trivial.action}] \label{thm:bir.trivial.action}
Any symplectic birational transformation of finite order on a manifold of type~\(\OG6\) acts trivially on the discriminant group \(\bL^\sharp\) of the lattice \(\bL\).
\end{theorem}

We remark that nonsymplectic automorphisms of manifolds of type \(\OG6\) can indeed have nontrivial action on both \(\bL\) and \(\bL^\sharp\), as shown by the first author~\cite{Grossi:non-sp.aut.OG6}.

Finally, our third theorem classifies all possible isometries on \(\bL\) induced by symplectic birational transformations of finite order on manifolds of type \(\OG6\). We denote the order of an isometry \(g \in \OO(\bL)\) by~\(|g|\) and we use the notation \(\bL^g\) and \(\bL_g\) for the invariant and coinvariant sublattices, respectively.

\begin{theorem}[\autoref{proof:thm:main.sp.OG6}] \label{thm:main.sp.OG6} 
An isometry \(g \in \OO(\bL)\) of finite order is induced by a symplectic birational transformation on a manifold of type \(\OG6\) if and only if
\begin{equation} \label{eq:possible.orders}
    |g| \in \set{1, 2, 3, 4, 5, 6, 8, 10, 12}
\end{equation}
and the pair \((\bL^g, \bL_g)\) appears in \autoref{tab:sp.OG6}. 
In this case, the following equality holds:
\begin{equation} \label{eq:cond.det}
    |{\det(\bL^g)}| = |{\det(\bL)} \cdot {\det(\bL_g)}|.
\end{equation}
\end{theorem}

Heuristically, given an action of a group \(G\) of birational transformations on an irreducible holomorphic symplectic manifold (not necessarily of type \(\OG6\)), a coinvariant sublattice of smaller rank corresponds to a larger family of manifolds admitting such an action. 
More precisely, the number of moduli is given by \(\rank(\bL^G) - 2\) (see \cite[§4]{Boissiere.Camere.Sarti:classif.aut.p-elem.lattices}).
For this reason, in \autoref{tab:sp.OG6} the pairs \((\bL^g,\bL_g)\) are ordered by increasing \(\rank(\bL_g)\).
The column `example' provides a matrix relative to a standard system of generators of \(3\bU\oplus 2[-2]\) representing an isometry \(g\) with given pair \((\bL^g,\bL_g)\).

The existence of birational transformations with given (co)invariant sublattice in \autoref{thm:main.sp.OG6} is ensured by the Torelli theorem for manifolds of type \(\OG6\) (\autoref{thm:torelli.OG6}). It would be very interesting to produce explicit examples of symplectic birational transformations inducing the isometries described in \autoref{tab:sp.OG6}.
Although there are ways to construct nonsymplectic automorphisms geometrically (see for instance \cite{Grossi:induced}), the lack of explicit families of manifolds of type \(\OG6\), beside desingularized moduli spaces of sheaves, makes this task difficult. We plan to study a special class of examples in a future work.

\subsection{Historical note}\label{Related works}
For a K3 surface \(X\) the groups \(\Aut(X)\) and \(\Bir(X)\) coincide.
All possible finite groups~\(G\) of symplectic automorphisms were classified by Nikulin~\cite{Nikulin:finite.aut.groups.K3.announcement, Nikulin:finite.aut.groups.K3} in the abelian case and then characterized by Mukai~\cite{Mukai:finite.groups.Mathieu.K3} in the general case as certain subgroups of the Mathieu group~\(\mathrm{M}_{23}\). 
Alternative proofs were given later by Xiao~\cite{Xiao:galois.covers.K3} and Kond\=o~\cite{Kondo:Niemeier.Mathieu.sp.aut.K3}. 
The invariant and coinvariant sublattices for \(G\) of prime order were classified by Morrison~\cite{Morrison:K3.large.picard} and by Garbagnati and Sarti~\cite{Garbagnati.sarti:K3.symplectic.prime}. 
Their classification was extended to all other groups by Hashimoto~\cite{Hashimoto:finite.symplectic.K3}. 
(For a survey, see \cite[Chapter~15]{Huybrechts:lectures.K3} or \cite{Kondo:survey.sp.K3}).

Finite groups of symplectic automorphisms on \(\K3^{[2]}\) fourfolds were studied first by Camere~\cite{Camere:symplectic.involutions.K3^2} and Mongardi~\cite{Mongardi:symplectic.involutions.K3^2, Mongardi:symplectic.aut.K3^2} in the case of prime order, and then by Höhn and Mason \cite{Hoehn.Mason:finite.groups.sp.aut.K3^2} in the general case.
Finite groups of symplectic automorphisms on manifolds of type~\(\K3^{[n]}\) with \(n > 2\) were treated by Huybrechts~\cite{Huybrechts:derived.cat.K3.sp.aut.Conway.group} and Mongardi~\cite{Mongardi:natural.deformations.symplectic.K3^n, Mongardi:towards.classification.symp.K3^n},
who in particular conjectured that all of them appear as certain subgroups of the Conway group~\(\mathrm{Co}_1\).
Partial results on symplectic automorphisms of \(\Kum^{[n]}\) manifolds were obtained by Mongardi, Tari and Wandel \cite{Mongardi.Tari.Wandel:kummer.fourfolds}.

\subsection{Contents of the paper} The definitions and the basic results used in this paper are recalled in \autoref{sec:preliminaries}, where we also explain how the geometrical problem is translated into an arithmetical one thanks to the Torelli theorem for irreducible holomorphic symplectic manifolds. 
The proofs of the statements above are carried out in \autoref{sec:sp.OG6}.

\subsection*{Acknowledgments}

We wish to warmly thank Fabio Bernasconi, Samuel Boissière, Simon Brandhorst, Chiara Camere, Alberto Cattaneo, Christian Lehn and Giovanni Mongardi for sharing their insights with us. A special acknowledgment goes to Maxim Smirnov, without whom this paper would probably not exist.

\begin{center}
\begin{longtable}{llll}
 \caption{Isometries \(g \in \OO(\bL)\) induced by symplectic birational transformations of finite order on manifolds of type~\(\OG6\).}
 \label{tab:sp.OG6} \\
 
 \toprule
 \(|g|\) & \(\bL^g\) & \(\bL_g\) & example \\ 
 \midrule
 \endfirsthead

 \multicolumn{4}{c}%
 {\tablename\ \thetable{}, follows from previous page} \\
 \midrule
 \(|g|\) & \(\bL^g\) & \(\bL_g\) & example \\ 
 \midrule
 \endhead

 \midrule
 \multicolumn{4}{c}{Continues on next page} \\
 \endfoot

 \bottomrule
 \endlastfoot
 
 \(1\) & \(\bL\) & \(0\) & \(\id\) \\
 \midrule
 \(2\) & \(2\bU \oplus [2] \oplus 2[-2]\) & \([-2]\) & {\tiny \(\begin{pmatrix} 1 & 0 & 0 & 0 & 0 & 0 & 0 & 0 \\ 0 & 1 & 0 & 0 & 0 & 0 & 0 & 0 \\ 0 & 0 & 1 & 0 & 0 & 0 & 0 & 0 \\ 0 & 0 & 0 & 1 & 0 & 0 & 0 & 0 \\ 0 & 0 & 0 & 0 & 0 & 1 & 0 & 0 \\ 0 & 0 & 0 & 0 & 1 & 0 & 0 & 0 \\ 0 & 0 & 0 & 0 & 0 & 0 & 1 & 0 \\ 0 & 0 & 0 & 0 & 0 & 0 & 0 & 1 \end{pmatrix}\)} \\
 \(2\) & \(\bU \oplus 2[2] \oplus 2[-2]\) & \(2[-2]\) & {\tiny \(\begin{pmatrix} 1 & 0 & 0 & 0 & 0 & 0 & 0 & 0 \\ 0 & 1 & 0 & 0 & 0 & 0 & 0 & 0 \\ 0 & 0 & 0 & 1 & 0 & 0 & 0 & 0 \\ 0 & 0 & 1 & 0 & 0 & 0 & 0 & 0 \\ 0 & 0 & 0 & 0 & 0 & 1 & 0 & 0 \\ 0 & 0 & 0 & 0 & 1 & 0 & 0 & 0 \\ 0 & 0 & 0 & 0 & 0 & 0 & 1 & 0 \\ 0 & 0 & 0 & 0 & 0 & 0 & 0 & 1 \end{pmatrix}\)} \\
 \(2\) & \(3[2] \oplus 2[-2]\) & \(3[-2]\) & {\tiny \(\begin{pmatrix} 0 & 1 & 0 & 0 & 0 & 0 & 0 & 0 \\ 1 & 0 & 0 & 0 & 0 & 0 & 0 & 0 \\ 0 & 0 & 0 & 1 & 0 & 0 & 0 & 0 \\ 0 & 0 & 1 & 0 & 0 & 0 & 0 & 0 \\ 0 & 0 & 0 & 0 & 0 & 1 & 0 & 0 \\ 0 & 0 & 0 & 0 & 1 & 0 & 0 & 0 \\ 0 & 0 & 0 & 0 & 0 & 0 & 1 & 0 \\ 0 & 0 & 0 & 0 & 0 & 0 & 0 & 1 \end{pmatrix}\)} \\
 \(2\) & \(3[2] \oplus [-2] \) & \(\bD_4\) & {\tiny \(\begin{pmatrix} 3 & 4 & 2 & -2 & 2 & -2 & 2 & 0 \\ 4 & 3 & 2 & -2 & 2 & -2 & 2 & 0 \\ -2 & -2 & -1 & 2 & -1 & 1 & -1 & 0 \\ 2 & 2 & 2 & -1 & 1 & -1 & 1 & 0 \\ -2 & -2 & -1 & 1 & -1 & 2 & -1 & 0 \\ 2 & 2 & 1 & -1 & 2 & -1 & 1 & 0 \\ -4 & -4 & -2 & 2 & -2 & 2 & -3 & 0 \\ 0 & 0 & 0 & 0 & 0 & 0 & 0 & 1 \end{pmatrix}\)} \\
 \midrule
 \(3\) & \(2\bU \oplus [6] \oplus [-2]\) & \(\bA_2\) & {\tiny \(\begin{pmatrix} 1 & 1 & 1 & -1 & 0 & 0 & 0 & 0 \\ 1 & 0 & 0 & 0 & 0 & 0 & 0 & 0 \\ -1 & 0 & 0 & 1 & 0 & 0 & 0 & 0 \\ 1 & 0 & 1 & 0 & 0 & 0 & 0 & 0 \\ 0 & 0 & 0 & 0 & 1 & 0 & 0 & 0 \\ 0 & 0 & 0 & 0 & 0 & 1 & 0 & 0 \\ 0 & 0 & 0 & 0 & 0 & 0 & 1 & 0 \\ 0 & 0 & 0 & 0 & 0 & 0 & 0 & 1 \end{pmatrix}\)} \\
 \(3\) & \(\bU \oplus 2[6]\) & \(2\bA_2\) & {\tiny \(\begin{pmatrix} 1 & 1 & 1 & -1 & 0 & 0 & 0 & 0 \\ 1 & 0 & 0 & 0 & 0 & 0 & 0 & 0 \\ -1 & 0 & 0 & 1 & 0 & 0 & 0 & 0 \\ 1 & 0 & 1 & 0 & 0 & 0 & 0 & 0 \\ 0 & 0 & 0 & 0 & 1 & 1 & 1 & 0 \\ 0 & 0 & 0 & 0 & 1 & 0 & 0 & 0 \\ 0 & 0 & 0 & 0 & -2 & 0 & -1 & 0 \\ 0 & 0 & 0 & 0 & 0 & 0 & 0 & 1 \end{pmatrix}\)} \\
 \midrule
  \(4\) & \(\bU \oplus [2] \oplus [4] \oplus [-2]\) & \(\bA_3\) & {\tiny \(\begin{pmatrix} -1 & 2 & -1 & -1 & -1 & -1 & 0 & 0 \\ 2 & -1 & 1 & 1 & 1 & 1 & 0 & 0 \\ 1 & -1 & 1 & 1 & 1 & 0 & 0 & 0 \\ 1 & -1 & 1 & 1 & 0 & 1 & 0 & 0 \\ 1 & -1 & 0 & 1 & 1 & 1 & 0 & 0 \\ 1 & -1 & 1 & 0 & 1 & 1 & 0 & 0 \\ 0 & 0 & 0 & 0 & 0 & 0 & 1 & 0 \\ 0 & 0 & 0 & 0 & 0 & 0 & 0 & 1 \end{pmatrix}\)} \\
 \(4\) & \(3[2] \oplus [-2]\) & \(\bD_4\) &  {\tiny \(\begin{pmatrix} 0 & 1 & -1 & -1 & 0 & 0 & -1 & 0 \\ 1 & 0 & 1 & 1 & 0 & 0 & 1 & 0 \\ 1 & -1 & 3 & 2 & 1 & -1 & 2 & 0 \\ 1 & -1 & 2 & 1 & 0 & 0 & 1 & 0 \\ -1 & 1 & -2 & -1 & 0 & 1 & -1 & 0 \\ 1 & -1 & 2 & 1 & 1 & 0 & 1 & 0 \\ 0 & 0 & -2 & 0 & 0 & 0 & -1 & 0 \\ 0 & 0 & 0 & 0 & 0 & 0 & 0 & 1 \end{pmatrix}\)} \\
 \(4\) & \(3[2] \oplus [-4]\) & \(\bA_3 \oplus [-2]\) & {\tiny \(\begin{pmatrix} 0 & 1 & 0 & 0 & 0 & 0 & 0 & 0 \\ 1 & 0 & 0 & 0 & 0 & 0 & 0 & 0 \\ 0 & 0 & 0 & 1 & -1 & -1 & -1 & 0 \\ 0 & 0 & 1 & 0 & 1 & 1 & 1 & 0 \\ 0 & 0 & 1 & -1 & 1 & 2 & 1 & 0 \\ 0 & 0 & 0 & 0 & 1 & 1 & 1 & 0 \\ 0 & 0 & 0 & 0 & 0 & -2 & -1 & 0 \\ 0 & 0 & 0 & 0 & 0 & 0 & 0 & 1 \end{pmatrix}\)} \\
 \(4\) & \(2[2] \oplus [4]\) & \(\bD_5\) & {\tiny \(\begin{pmatrix} 3 & 4 & 2 & 2 & 2 & 2 & 2 & 4 \\ 4 & 3 & 2 & 2 & 2 & 2 & 2 & 4 \\ 2 & 2 & 2 & 2 & 1 & 2 & 1 & 3 \\ 2 & 2 & 2 & 2 & 2 & 1 & 1 & 3 \\ 2 & 2 & 2 & 1 & 2 & 2 & 1 & 3 \\ 2 & 2 & 1 & 2 & 2 & 2 & 1 & 3 \\ -4 & -4 & -2 & -2 & -2 & -2 & -3 & -4 \\ -8 & -8 & -6 & -6 & -6 & -6 & -4 & -11 \end{pmatrix}\)} \\
 \midrule
 \(5\) & \(\bU \oplus \begin{pmatrix} 4 & 2 \\ 2 & 6 \end{pmatrix}\) & \(\bA_4\) & {\tiny \(\begin{pmatrix} 1 & 0 & 0 & 0 & 0 & 0 & 0 & 0 \\ 0 & 1 & 0 & 0 & 0 & 0 & 0 & 0 \\ 0 & 0 & 0 & 2 & 2 & 1 & 1 & 1 \\ 0 & 0 & 1 & 0 & -1 & -1 & 0 & -1 \\ 0 & 0 & -1 & 2 & 2 & 2 & 1 & 1 \\ 0 & 0 & 0 & 1 & 2 & 1 & 1 & 1 \\ 0 & 0 & 0 & -2 & -2 & 0 & -1 & 0 \\ 0 & 0 & 0 & -2 & -2 & -2 & -2 & -1 \end{pmatrix}\)} \\
 \midrule
 \(6\) & \(\bU \oplus [2] \oplus [6] \oplus [-2]\) & \(\bA_2 \oplus [-2]\) & {\tiny \(\begin{pmatrix} 0 & 1 & 0 & 0 & 0 & 0 & 0 & 0 \\ 1 & 0 & 0 & 0 & 0 & 0 & 0 & 0 \\ 0 & 0 & 0 & 1 & -1 & 0 & 0 & 0 \\ 0 & 0 & 1 & 0 & 1 & 0 & 0 & 0 \\ 0 & 0 & 1 & -1 & 1 & 1 & 0 & 0 \\ 0 & 0 & 0 & 0 & 1 & 0 & 0 & 0 \\ 0 & 0 & 0 & 0 & 0 & 0 & 1 & 0 \\ 0 & 0 & 0 & 0 & 0 & 0 & 0 & 1 \end{pmatrix}\)} \\
 \(6\) & \(3[2] \oplus [-2]\) & \(\bD_4\) & {\tiny \(\begin{pmatrix} 4 & 3 & 2 & -2 & 2 & -2 & 2 & 0 \\ 3 & 3 & 1 & -1 & 2 & -2 & 2 & 0 \\ -2 & -1 & 0 & 1 & -1 & 1 & -1 & 0 \\ 2 & 1 & 1 & 0 & 1 & -1 & 1 & 0 \\ -2 & -2 & -1 & 1 & -1 & 2 & -1 & 0 \\ 2 & 2 & 1 & -1 & 2 & -1 & 1 & 0 \\ -4 & -4 & -2 & 2 & -2 & 2 & -3 & 0 \\ 0 & 0 & 0 & 0 & 0 & 0 & 0 & 1 \end{pmatrix}\)} \\
 \(6\) & \(2[2] \oplus [6] \oplus [-2]\) & \(\bA_2 \oplus 2[-2]\) & {\tiny \(\begin{pmatrix} 0 & 1 & 0 & 0 & 0 & 0 & 0 & 0 \\ 1 & 0 & 0 & 0 & 0 & 0 & 0 & 0 \\ 0 & 0 & 0 & 1 & 0 & 0 & 0 & 0 \\ 0 & 0 & 1 & 0 & 0 & 0 & 0 & 0 \\ 0 & 0 & 0 & 0 & 0 & 1 & 0 & 0 \\ 0 & 0 & 0 & 0 & 1 & 1 & 1 & 0 \\ 0 & 0 & 0 & 0 & 0 & -2 & -1 & 0 \\ 0 & 0 & 0 & 0 & 0 & 0 & 0 & 1 \end{pmatrix}\)} \\
 \(6\) & \(2[2] \oplus [6]\) & \(2\bA_2\oplus [-2]\) & {\tiny \(\begin{pmatrix} 0 & 1 & 0 & 0 & 0 & 0 & 0 & 0 \\ 1 & 0 & 0 & 0 & 0 & 0 & 0 & 0 \\ 0 & 0 & 0 & 1 & 0 & 0 & 0 & 0 \\ 0 & 0 & 1 & 1 & 0 & 0 & 1 & 0 \\ 0 & 0 & 0 & 0 & 0 & 1 & 0 & 0 \\ 0 & 0 & 0 & 0 & 1 & 1 & 0 & 1 \\ 0 & 0 & 0 & -2 & 0 & 0 & -1 & 0 \\ 0 & 0 & 0 & 0 & 0 & -2 & 0 & -1 \end{pmatrix}\)} \\
 \midrule
 \(8\) & \(2[2] \oplus [4]\) & \(\bD_5\) & {\tiny \(\begin{pmatrix} 2 & 1 & 1 & 1 & 2 & 1 & 1 & 2 \\ 2 & 2 & 1 & 2 & 2 & 2 & 1 & 3 \\ 1 & 1 & 1 & 2 & 2 & 1 & 1 & 2 \\ 2 & 1 & 2 & 2 & 2 & 2 & 1 & 3 \\ 2 & 1 & 1 & 2 & 3 & 3 & 2 & 3 \\ 2 & 1 & 1 & 2 & 3 & 2 & 1 & 3 \\ -2 & 0 & 0 & -2 & -2 & -2 & -1 & -2 \\ -6 & -4 & -4 & -6 & -8 & -6 & -4 & -9 \end{pmatrix}\)} \\
 \midrule
 \(10\) & \([2] \oplus \begin{pmatrix} 4 & 2 \\ 2 & 6 \end{pmatrix}\) & \(\bA_4 \oplus [-2]\) & {\tiny \(\begin{pmatrix} 0 & 1 & 0 & 0 & 0 & 0 & 0 & 0 \\ 1 & 0 & 0 & 0 & 0 & 0 & 0 & 0 \\ 0 & 0 & 1 & 1 & 1 & 0 & 1 & 0 \\ 0 & 0 & 1 & 0 & -1 & 0 & 0 & 0 \\ 0 & 0 & 0 & 1 & 2 & 1 & 1 & 1 \\ 0 & 0 & 0 & 0 & 1 & 0 & 0 & 0 \\ 0 & 0 & -2 & 0 & 0 & 0 & -1 & 0 \\ 0 & 0 & 0 & 0 & -2 & 0 & 0 & -1 \end{pmatrix}\)} \\
 \midrule
 \(12\) & \(2[2] \oplus [4]\) & \(\bD_5\) & {\tiny \(\begin{pmatrix} 2 & 2 & 2 & 1 & 2 & 2 & 1 & 3 \\ 2 & 2 & 1 & 2 & 2 & 2 & 1 & 3 \\ 1 & 2 & 2 & 2 & 2 & 2 & 1 & 3 \\ 2 & 1 & 2 & 2 & 2 & 2 & 1 & 3 \\ 2 & 2 & 2 & 2 & 3 & 3 & 1 & 4 \\ 2 & 2 & 2 & 2 & 3 & 4 & 2 & 4 \\ -2 & -2 & -2 & -2 & -2 & -4 & -1 & -4 \\ -6 & -6 & -6 & -6 & -8 & -8 & -4 & -11 \end{pmatrix}\)} \\
\end{longtable}
\end{center}

\section{Preliminaries} \label{sec:preliminaries}

This section not only explains our conventions and fixes the notation used throughout the paper, but it also reviews some minor results -- certainly known to experts -- for which we could not find a satisfactory treatment in the literature.

In \autoref{subsec:lattices} we introduce our conventions on lattices, following Nikulin \cite{Nikulin:int.sym.bilinear.forms} and Miranda and Morrison \cite{Miranda.Morrison}. 
In \autoref{subsec:prim.emb} we explain the algorithm, based on Nikulin's work, that we use to find all primitive embeddings of two given lattices.
In \autoref{subsec:isometries} we collect various results related to lattice isometries.
In \autoref{subsec:m-elementary} we introduce the notion of \(m\)-elementary lattice and we explain how to enumerate such lattices.
In \autoref{subsec:Hodge.structures} we recall the definition of Hodge structure and other related concepts, as given by Morrison \cite{Morrison:K3.large.picard} (also cf. \cite[§3.1]{Huybrechts:lectures.K3}). Finally, in \autoref{subsec:OG6} we state the Torelli theorem in the case of manifolds of type \(\OG6\), which is the key to characterize groups of automorphisms and birational transformations in an arithmetical way.

\subsection{Lattices} \label{subsec:lattices}
In this paper, a \emph{lattice} of \emph{rank} \(r\) is a free finitely generated \(\IZ\)-module \(L \cong \IZ^r\) endowed with a nondegenerate symmetric bilinear pairing \(L \otimes L \rightarrow \IZ\) denoted \(e \otimes f \mapsto e\cdot f\). 
We write \(e^2 = e \cdot e\). We say that \(L\) is \emph{even} if \(e^2 \in 2\IZ\) for each \(e \in L\).
The lattice \(L(n)\) is obtained by composing the pairing with multiplication by \(n \in \IZ \setminus \set{0}\).
The \emph{rank}, \emph{signature} and \emph{determinant} of \(L\) are denoted \(\rank(L)\), \(\sign(L)\), \(\det(L)\), respectively. 
A lattice is \emph{unimodular} if \(|{\det(L)}| = 1\), and it is \emph{hyperbolic} if \(\sign(L) = (1,\rank(L)-1)\).

The \emph{dual} of \(L\) is defined as 
\(
    L^\vee = \Set{x \in L\otimes \IQ}{x \cdot y \in \IZ \text{ for each \(y \in L\)}}. 
\)
The group 
\[
    L^\sharp \coloneqq L^\vee/L
\]
is called the \emph{discriminant group} of~\(L\). It is a finite abelian group of order \(|L^\sharp| = |{\det(L)}|\). 
For even~\(L\), the \(\IQ\)-linear extension of the pairing on \(L\) endows \(L^\sharp\) with a \(\IQ/2\IZ\)-valued quadratic form.
There is a natural homomorphism \(\OO(L) \rightarrow \OO(L^\sharp)\). 
The image of a subgroup \(G \subset \OO(L)\) through this homomorphism is denoted \(G^\sharp\). 
The \emph{divisibility} of a vector \(v \in L\) is defined as
\[
    (v, L) \coloneqq \gcd\Set{v \cdot w}{w \in L}.
\]
Note that \(v/(v, L)\) defines an element of \(L^\vee\), hence of \(L^\sharp\). 

The unimodular indefinite even lattice of rank \(2\) is denoted \(\bU\). The negative definite ADE lattices are denoted \(\bA_n,\bD_n,\bE_n\). 
The notation \([m]\) with \(m \in \IZ \setminus \set{0}\) denotes a lattice of rank \(1\) generated by a vector of square \(m\).

Two even lattices \(L_1, L_2\) are said to belong to the same \emph{genus} if \(\sign(L_1) = \sign(L_2)\) and \(L_1^\sharp \cong L_2^\sharp\) as finite quadratic forms (cf. \cite[Corollary 1.9.4]{Nikulin:int.sym.bilinear.forms}).
All lattices appearing in this paper are unique in their genus. For indefinite lattices this follows from Nikulin's or Miranda and Morrison's criteria (\cite[Theorem~3.6.2]{Nikulin:int.sym.bilinear.forms}, \cite[Theorem~1.14.2]{Nikulin:int.sym.bilinear.forms}, \cite[Corollary~7.8]{Miranda.Morrison}), and for definite lattices from the Smith--Minkowski--Siegel formula (see~\cite{Conway.Sloane:mass.formula}).

\subsection{Primitive embeddings} \label{subsec:prim.emb} 
An embedding of lattices \(M \subset L\) is called \emph{primitive} if \(L/M\) is a free abelian group. In this case we write \(M \hookrightarrow L\) and we denote by \(N \coloneqq M^\perp\) the \emph{orthogonal complement} of \(M\) in \(L\). We refer to \cite[Proposition~1.5.1 and Proposition~1.15.1]{Nikulin:int.sym.bilinear.forms} for more details on what follows.

A primitive embedding of even lattices \(M \hookrightarrow L\) is given by a \emph{gluing subgroup} \(H \subset M^\sharp\) and a \emph{gluing isometry} \(\gamma \colon H \rightarrow H' \subset N^\sharp\).
If \(\Gamma\) denotes the \emph{gluing graph} of \(\gamma\) in \(M^\sharp \oplus N(-1)^\sharp\), the following identification between finite quadratic forms holds (the quadratic form on the right hand side is the one induced by \(M^\sharp \oplus N(-1)^\sharp\)):
\begin{equation} \label{eq:L^sharp=Gamma^perp/Gamma}
  L^\sharp \cong \Gamma^\perp/\Gamma.  
\end{equation}
Two primitive embeddings \(M \hookrightarrow L\) are equivalent under the action of \(\OO(M)\) and \(\OO(L)\) if and only if the corresponding groups \(H\) and \(H'\) are conjugate under the action of \(\OO(M)\) and \(\OO(N(-1)) = \OO(N)\) in a way that commutes with the gluing isometries.

Assume now that \(L\) is unique in its genus. By \cite[Proposition~1.15.1]{Nikulin:int.sym.bilinear.forms}, an embedding \(M \hookrightarrow L\) is equivalently given by a subgroup \(K \subset L^\sharp\), which here we call \emph{embedding subgroup}, and an isometry \(\xi\colon K \rightarrow K' \subset M^\sharp\).
If \(\Xi\) denotes the graph of \(\xi\) in \(L^\sharp \oplus M(-1)^\sharp\), the following identification between finite quadratic forms holds (the quadratic form on the right hand side is the one induced by \(L^\sharp \oplus M(-1)^\sharp\)):
\begin{equation} \label{eq:(S^perp)^sharp=Xi^perp/Xi}
    N^\sharp \cong \Xi^\perp/\Xi.
\end{equation}

We define the \emph{gluing} and \emph{embedding index} to be \(h \coloneqq |H|\) and \(k \coloneqq |K|\), respectively. They satisfy
\begin{equation} \label{eq:gluing.embedding.index}
    h^2 \cdot |{\det(L)}| = |{\det(M)} \cdot {\det(N)}|, \qquad k^2 \cdot |{\det(N)}| = |{\det(L)} \cdot {\det(M)}|,
\end{equation}
from which it follows that \(hk = |{\det(M)}|\).

Given \(M\) and \(L\), we proceed in the following way in order to find all primitive embeddings \(M \hookrightarrow L\) up to the action of \(\OO(L)\).
\begin{enumerate}[(i)]
    \item We determine the embedding index \(k\) from \eqref{eq:gluing.embedding.index}.
    \item We find all embedding subgroups \(K \subset L^\sharp\) with an isometric counterpart \(K' \subset M^\sharp\).
    \item For each isometry \(\xi\colon K\rightarrow K'\) we compute \(\sign(M^\perp)\) and \((M^\perp)^\sharp\) from \eqref{eq:(S^perp)^sharp=Xi^perp/Xi}, determining the possible genera for \(M^\perp\).
    \item We find all possible \(N = M^\perp\) in these genera.
    \item For each \(N\), we compute the gluing index \(h\) of \(M \hookrightarrow L\) with \(N \cong M^\perp\).
    \item We find all gluing subgroups \(H \subset M^\sharp\) with an isometric counterpart \(H' \subset N(-1)^\sharp\) up to the action of \(\OO(M)\) and \(\OO(N(-1)) = \OO(N)\).
\end{enumerate}

Given a vector \(v \in M\), we are interested in computing \((v, L)\). 
Of course, \((v, L)\) depends on the primitive embedding \(M \hookrightarrow L\). In general, it only holds \((v, L) \mid (v, M)\).
The following lemma provides a useful formula.

\begin{lemma} \label{lem:L.div.v}
Let \(M,L\) be two lattices. If a primitive embedding \(M \hookrightarrow L\) is defined by the gluing subgroup \(H \subset M^\sharp\), then for each \(v \in M\) it holds
\[
    (v, L) = \max\Set{d \in \IN}{v/d \in H^\perp}.
\]
\end{lemma}
\proof
Let \(\Gamma\) be the gluing graph. The assertion follows from the identification~\eqref{eq:L^sharp=Gamma^perp/Gamma} and the elementary observation that \((v,L) = \max\Set{d \in \IN}{v/d \in L^\sharp}\), so
\[ 
    (v,L) = \max\Set{d \in \IN}{(v/d,0) \in \Gamma^\perp} = \max\Set{d \in \IN}{v/d \in H^\perp}. \qedhere
\]
\endproof 

\begin{corollary} \label{cor:L.div.v.easy.case}
Let \(M \hookrightarrow L \) be a primitive embedding. If \(|{\det(M^\perp)}| = |{\det(L)} \cdot {\det(M)}|\), then \((v, L) = 1\) for each \(v \in M\).
\end{corollary}
\proof
Indeed, by \eqref{eq:gluing.embedding.index} the gluing index is equal to \(|{\det(M)}|\), hence \(H^\perp\) is trivial. 
\endproof 

\subsection{Isometries} \label{subsec:isometries}
Given a lattice \(L\) and a subgroup \(G \subset \OO(L)\), we denote by~\(L^G\) the \emph{invariant sublattice}. Its orthogonal complement \(L_G \coloneqq (L^G)^\perp \subset L\) is called the \emph{coinvariant sublattice}. 
Given \(g \in \OO(L)\), we denote by \(L^g\) and \(L_g\) the invariant and coinvariant sublattices, respectively, of the subgroup generated by \(g\).

We adopt the convention that the (real) \emph{spinor norm}
\[
\spin \colon \OO(L) \rightarrow \IR^\times/(\IR^\times)^2 = \set{\pm 1}
\]
takes the value \(+1\) on a reflection \(\rho_v\) induced by a vector \(v\) with \(v^2 < 0\). (We refer the reader to \cite[Chapter~I.10]{Miranda.Morrison}, although there the opposite convention is used.) The kernel of the spinor norm is denoted \(\OO^+(L)\). Equivalently, an isometry \(g \in \OO(L)\) belongs to~\(\OO^+(L)\) if and only if \(g\) preserves the orientation of a positive definite subspace \(V \subset L \otimes \IR\) of maximal dimension (cf. \cite[Lemma~4.1]{Markman:survey} and \cite[Chapter~I, Proposition~11.3]{Miranda.Morrison}).

\begin{lemma}\label{lem:spin}
If \(L_G\) is negative definite, then \(G \subset \OO^+(L)\).
\end{lemma}
\proof
For any \(g \in G\) it holds \(L_g \subset L_G\), therefore \(L_g\) is negative definite. 
The isometry \(g\), seen as an element of \(\OO(L \otimes \IR)\), can be written as the product of reflections \(\rho_{v}\) for certain vectors \(v \in (L\otimes \IR)_g\). 
It follows that \(v^2 < 0\) for each such \(v\), so \(\spin(g) = +1\).
\endproof

\begin{lemma} \label{lem:prime.ord}
If \(g \in \OO(L)\) has prime order \(p\), then \(p - 1 \mid \rank(L_g)\).
\end{lemma}
\proof
Since \(g\) is defined over \(\IZ\), its eigenspaces in \(L \otimes \IC\) associated with the primitive \(p\)th roots of unity have the same dimension.
\endproof 

\subsection{\texorpdfstring{\(m\)}{m}-elementary lattices} \label{subsec:m-elementary}

Given \(m \in \IN\), we say that an abelian group \(A\) is \emph{\(m\)-elementary} if \(m\alpha = 0\) for each \(\alpha \in A\).
The length of \(A\) of an abelian group is denoted by~\(\ell(A)\), and its \(p\)-length by \(\ell_p(A)\).
We say that a lattice \(L\) is \emph{\(m\)-elementary} if \(L^\sharp\) is an \(m\)-elementary abelian group. 
An \(m\)-elementary lattice \(L\) satisfies
\begin{equation} \label{eq:bound.m-elementary}
    |{\det(L)}| \leq m^{\ell(L^\sharp)} \leq m^{\rank(L)}.
\end{equation}

The following lemma is usually stated for \(m\) a prime number, but the same proof works for general \(m\) (see for instance \cite[§5.3]{Boissiere.NieperWisskirchen.Sarti:smith.theory}). The details are left to the reader.

\begin{lemma} \label{lem:unimodular.m-el}
If \(\Lambda\) is a unimodular lattice and \(f \in \OO(\Lambda)\) has order~\(m\), then \(\Lambda^f\) and \(\Lambda_f\) are \(m\)-elementary lattices. \qed
\end{lemma}

The following statements are a direct consequence of \eqref{eq:(S^perp)^sharp=Xi^perp/Xi} and \cite[Lemma~3.1]{Brandhorst.Sonel.Veniani:idoneal.genera}.

\begin{lemma} \label{lem:embedding.m-elementary}
If \(M\) is an \(m\)-elementary lattice, \(L\) is an \(l\)-elementary lattice and there exists a primitive embedding \(M \hookrightarrow L\), then \(N\) is \(\lcm(m,l)\)-elementary. \qed 
\end{lemma}

\begin{lemma} \label{lem:Trieste}
Let \(L\) be an \(m\)-elementary even lattice. If \(m\) is odd, then \(\rank(L)\) is even. \qed
\end{lemma} 

Given \(m, d \in \IN\) and \(s_+, s_- \in \IZ_{\geq 0}\) we can enumerate all \(m\)-elementary lattices \(L\) of signature \((s_+,s_-)\) and determinant \(|{\det(L)}| \leq d\) in the following way.
\begin{enumerate}
    \item We list all possible \(m\)-elementary abelian groups \(A\) of order \(|A| \leq d\).
    \item We list all possible finite quadratic forms \(q\colon A \rightarrow \IQ/2\IZ\) using Miranda and Morrison's classification~\cite{Miranda.Morrison}.
    \item \label{step:3} For each \(q\), we consider the genus \(\fg\) of signature \((s_+,s_-)\) and discriminant group~\(q\).
    \item We list all lattices in \(\fg\).
\end{enumerate}
We carried out Step~\ref{step:3} using the \texttt{sageMath} function \texttt{genera} found in \path{sage.quadratic_forms.genera.genus} and implemented by Brandhorst~\cite{sage}.
For all computations in this paper, the last step is made easier by the fact that all lattices considered are unique in their genus.

Now let us consider a primitive embedding \(L \hookrightarrow \Lambda\) with \(\Lambda\) a unimodular lattice. We write \(R \coloneqq L^\perp\) and we fix isomorphism \(L^\sharp \cong R(-1)^\sharp\) and \(\OO(L^\sharp) \cong \OO(R^\sharp)\). 
Given \(f \in \OO(\Lambda)\) with \(f(L) = L\), we can consider its restrictions \(g = f|_L\) and \(h = f|_R\). 
Then, \(g^\sharp = h^\sharp\) through the isomorphism \(\OO(L^\sharp) \cong \OO(R^\sharp)\). 
Conversely, by \cite[Corollary~1.5.2]{Nikulin:int.sym.bilinear.forms}, if there exist \(g \in \OO(L)\) and \(h \in \OO(R)\) such that \(g^\sharp\) and \(h^\sharp\) coincide through the isomorphism \(\OO(L^\sharp) \cong \OO(R^\sharp)\), then there exists \(f \in \OO(\Lambda)\) with \(f(L) = L\), \(f|_L =g\) and \(f|_R = h\). 

\begin{lemma} \label{lem:coinv.m-elementary}
Let \(L\) be a lattice. If \(g \in \OO(L)\) has order \(m\) and \(g^\sharp = \id\), then \(L_g\) is an \(m\)-elementary lattice.
\end{lemma}
\proof
Let \(\Lambda\) be a unimodular lattice such that there exists a primitive embedding \(L \hookrightarrow \Lambda\). 
(There always exists one with \(\rank(\Lambda) \leq 2\rank(L)\), since the direct sum of \(\rank(L)\) copies of \(\bU\) is an overlattice of \(L \oplus L(-1)\).)
Let \(R = L^\perp\). Since \(g^\sharp = \id\), we can extend \(g\) to an isometry \(f \in \OO(\Lambda)\) by \(h = \id \in \OO(R)\). 
As \(R_h = 0\), it holds \(L_g \cong \Lambda_f\). Given that \(|f| = |g|\), we deduce from \autoref{lem:unimodular.m-el} that \(L_g\) is \(m\)-elementary.
\endproof

\begin{lemma} \label{lem:cond.det}
Let \(L\) be a lattice and \(g \in \OO(L)\) an isometry. Consider the restriction \(g' = g|_{L_g} \in \OO(L_g)\). If 
\[
    |{\det(L^g)}| = |{\det(L)\cdot \det(L_g)}|,
\] 
then \(g^\sharp = \id\) and \((g')^\sharp = \id\).
\end{lemma}
\proof
Indeed, by \eqref{eq:gluing.embedding.index} the equation holds if and only if the gluing subgroup of \(L_g \hookrightarrow L\) is the whole group \(L_g^\sharp\). Hence, \((L^g)^\sharp = L^\sharp \oplus L_g^\sharp\).
Obviously, the restriction of \(g\) to \(L^g\) is the identity, so its image in \(\OO((L^g)^\sharp)\) is trivial. Therefore, also \(g^\sharp \) and \((g')^\sharp\) must be trivial.
\endproof

The discriminant group \(L^\sharp\) of a lattice \(L\) is said have \emph{parity} \(\delta = 0\) if no element of order~\(2\) in \(L^\sharp\) has square \(\pm 1/2 \mod 2\IZ\) (cf. \cite[p.~144]{Nikulin:int.sym.bilinear.forms}). 
Otherwise, one says that \(L^\sharp\) has parity \(\delta = 1\).

The following theorem, which we will use in~\autoref{subsec:prime.order} to deal with isometries of order \(2\), is easily derived from Nikulin's classification of \(2\)-elementary lattices. 
An analogous theorem for groups of odd prime order \(p\) has been proven by Brandhorst and Cattaneo~\cite{Brandhorst.Cattaneo:prime.order.unimodular}. 

\begin{theorem} \label{thm:involutions.unimodular}
Let \(l_+,l_-,t_+, t_-, a \in \IZ_{\geq 0},\delta \in \set{0,1}\) such that \(l_+ - l_- \equiv 0 \mod 8\). There exist a unimodular even lattice \(\Lambda\) of signature \((l_+,l_-)\) and an involution \(f \in \OO(\Lambda)\) with coinvariant lattice \(\Lambda_f\) of signature~\((t_+,t_-)\), and discriminant group \(\Lambda_f^\sharp\) of length~\(a\) and parity~\(\delta\) if and only if the following conditions hold:
\begin{enumerate}[(1)]
    \item \(t_+ \leq l_+\) and \(t_- \leq l_-\);
    \item \(a \leq \min(t_+ + t_-, l_+ + l_- - t_+ -t_-)\);
    \item \(t_+ + t_- + a \equiv 0 \mod 2\);
    \item if \(\delta = 0\), then \(t_+ - t_- \equiv 0 \mod 4\);
    \item if \(a = 0\), then \(\delta = 0\) and \(t_+ - t_- \equiv 0 \mod 8\);
    \item if \(a = 1\), then \(t_+ - t_- \equiv \pm 1 \mod 8\);
    \item if \(a = 2\) and \(t_+ - t_- \equiv 4 \mod 8\), then \(\delta = 0\);
    \item if \(\delta = 0\), and \(a = t_+ + t_-\) or \(a = l_+ + l_- - t_+ - t_-\), then \(t_+ - t_- \equiv 0 \mod 8\).
\end{enumerate}
\end{theorem}
\proof 
If such an involution \(f \in \OO(\Lambda)\) exists, then it holds \(\sign(\Lambda_f) = (l_+ - t_+, l_- - t_-)\) and \((\Lambda^f)^\sharp \cong (\Lambda_f)^\sharp\). 
Nikulin's \cite[Theorem 3.6.2]{Nikulin:int.sym.bilinear.forms} then implies that the conditions listed hold. 

Conversely, if the conditions hold, then by the same theorem there exist \(2\)-elementary even lattices \(M,N\) with \(\sign(M) = (l_+ - t_+, l_- - t_-), \sign(N) = (t_+,t_-)\) and \(M^\sharp \cong N(-1)^\sharp\) of length \(a\) and parity \(\delta\).
Therefore, \(M \oplus N\) admits a primitive overlattice \(\Lambda\). 
Moreover, \(\id\) on \(M\) and \(-\id\) on \(N\) glue well to an involution \(f \in \OO(\Lambda)\) with \(M \cong \Lambda^f\) and \(N \cong \Lambda_f\). 
\endproof 

\subsection{Hodge structures} \label{subsec:Hodge.structures}

A \emph{Hodge structure of weight \(2\)} on a lattice \(L\) is the choice of three complex subspaces \(L^{2,0},L^{1,1},L^{0,2}\) of \(L \otimes \IC\) with \(L \otimes \IC = L^{2,0} \oplus L^{1,1} \oplus L^{0,2}\), \(\overline{L^{2,0}} = L^{0,2}\), \(\overline{L^{1,1}} = L^{1,1}\), and such that
\begin{align*} 
    v \cdot \overline v &> 0 \text{ for } v \in L^{2,0}, v \neq 0, \\
    v \cdot w &= 0 \text{ for } v,w \in L^{2,0}, \\
    v \cdot w &= 0 \text{ for } v \in L^{2,0} \oplus L^{0,2}, w \in L^{1,1}.
\end{align*}

If \(\dim_\IC L^{2,0} = 1\), the Hodge structure is said to be \emph{of \(\K3\) type}.
In this paper we will only consider Hodge structures of \(\K3\) type, which for the sake of simplicity we will just call Hodge structures.

An isometry of a lattice \(L\) with a Hodge structure is a \emph{Hodge isometry} if its linear extension to \(\IC\) preserves the Hodge structure. A Hodge isometry of \(L\) is a \emph{symplectic isometry} if its linear extension restricts to the identity on \(L^{2,0}\).
The subgroups of Hodge and symplectic isometries are denoted \(\OO_{\hodge}(L)\) and \(\OO_{\symplectic}(L)\), respectively. 

A \emph{signed Hodge structure} on \(L\) is a Hodge structure such that the restriction of the bilinear form to \(L^{1,1} \cap (L \otimes \IR)\) has signature \((1,\dim L^{1,1}-1)\), together with the choice of one connected component, called \emph{positive cone} and denoted \(\sP\), of the subset
\[
    \Set{v \in L^{1,1} \cap (L \otimes \IR)}{v^2 > 0}.
\]
The real subspace \((L^{2,0} \oplus L^{0,2}) \cap (L \otimes \IR)\) is always positive definite,
so only lattices of signature \((3,\rank L - 3)\) can admit a signed Hodge structure of weight \(2\). 

Given a cone \(C \subset L \otimes \IR\), we set
\[
    \OO(L, C) \coloneqq \Set{g \in \OO(L)}{g(C) = C}.
\]
Note that \(\OO^+(L) = \OO(L, \sP)\) if \(L\) is of signature \((3, \rank L -3)\).

\begin{lemma} \label{lem:convex.cone}
Let \(L\) be a lattice. If \(C \subset L \otimes \IR\) is a convex cone, and \(G \subset \OO(L, C)\) is a finite subgroup, then
\[
    (L \otimes \IR)^G \cap C \neq \emptyset.
\]
\end{lemma}
\proof
Take \(\omega \in C\) and define \(\omega_G \coloneqq \sum_{g \in G} g(\omega)\). Clearly, \(\omega_G \in (L \otimes \IR)^G\), but also \(\omega_G \in C\), as \(g(\omega) \in C\) for every \(g \in G\) and \(C\) is a convex cone.
\endproof 

Given a lattice \(L\) with a signed Hodge structure, we define \(\OO^+_\symplectic(L) \coloneqq \OO^+(L) \cap \OO_\symplectic(L).\)

\begin{lemma} \label{lem:coinvariant.negative.definite}
Let \(L\) be a lattice of signature \((3,r-3)\), with \(r \geq 3\), and \(G \subset \OO(L)\) a finite subgroup. 
If there exists a signed Hodge structure such that \(G \subset \OO^+_{\symplectic}(L)\), then \(L_G\) is negative definite.
\end{lemma}
\proof 
The invariant real subspace \((L \otimes \IR)^G\) contains the \(2\)-dimensional real subspace \((L^{2,0} \oplus L^{0,2}) \cap (L \otimes \IR)\), which is positive definite.
Moreover, \((L \otimes \IR)^G\) intersects the convex cone \(\sP \subset L^{1,1} \cap (L \otimes \IR)\) by \autoref{lem:convex.cone}. 
Therefore, \(\sign L^G = (3, \rank L^G - 3)\), and \(L_G\) is consequently negative definite.
\endproof

\begin{lemma} \label{lem:L^1,1=L_G}
Let \(L\) be a lattice of signature \((3,r-3)\), with \(r \geq 3\), and \(G \subset \OO(L)\) a finite subgroup. 
If \(L_G\) is negative definite, then there exists a signed Hodge structure on \(L\) such that \(L^{1,1} \cap L = L_G\) and \(G \subset \OO^+_\symplectic(L)\).
\end{lemma}
\proof[Proof]
Since \(L_G\) is negative definite, \(L^G \otimes \IR\) contains a positive definite subspace of maximal dimension~\(3\).
This implies that 
\[ 
    \dim_\IC (L\otimes \IC)^G = \dim_\IR (L \otimes \IR)^G \geq \dim_\IR (L^G \otimes \IR) \geq 3.
\]
Therefore, we can find \(\sigma \in (L \otimes \IC)^G\) with \(\sigma^2 = 0\), \(\sigma \cdot \overline \sigma > 0\). 
Choosing \(\sigma\) outside a certain countable union of hyperplanes, we can suppose that
\[
    \sigma^\perp \cap L = L_G.
\]
In other words, putting \(L^{2,0} = \IC\sigma, L^{0,2} = \overline{L^{2,0}}\) and \(L^{1,1} = (L^{2,0} \oplus L^{0,2})^\perp \subset L \otimes \IC\) we define a signed Hodge structure on \(L\) (choosing \(\sP\) arbitrarily) such that \(L^{1,1} \cap L = L_G\). In this way, \(G \subset \OO_{\symplectic}(L)\) by construction.
By \autoref{lem:spin}, it also holds \(G \subset \OO^+(L)\), therefore \(G \subset \OO^+_{\symplectic}(L)\).
\endproof

Fix a signed Hodge structure on \(L\). Given a subset \(\cW \subset L\) of primitive vectors which is invariant under \(\OO(L)\), a connected component of
\begin{equation} \label{eq:exc.chamb}
    \sP \setminus \bigcup_{w \in \cW \cap L^{1,1}} w^\perp 
\end{equation}
is called a \emph{chamber defined by \(\cW\)}. Note that every chamber is a convex cone and that \(\OO^+_\hodge(L)\) acts on the set of chambers. 
We put \(\OO^+_\symplectic(L, C) \coloneqq \OO^+_\symplectic(L) \cap \OO(L, C)\).

\begin{lemma} \label{lem:no.walls.in.coinv}
Let \(L\) be a lattice with a signed Hodge structure. If \(C\) is a chamber defined by a subset \(\cW \subset L\) and \(G \subset \OO^+_{\symplectic}(L, C)\) is a finite subgroup, then 
\[ L_G \cap \cW = \emptyset.\] 
\end{lemma}
\begin{proof}
By \autoref{lem:convex.cone}, there exists \(v \in (L \otimes \IR)^G \cap C\). Each \(w \in L_G\) is orthogonal to \(v\), hence \(w \notin \cW\).
\end{proof}

\subsection{Torelli theorem} \label{subsec:OG6}

An \emph{irreducible holomorphic symplectic} manifold is a simply connected compact Kähler manifold \(X\) such that \(\HH^0(X, \Omega^2_X)\) is generated by a nowhere degenerate holomorphic \(2\)-form \(\sigma\).
An automorphism or a birational transformation of \(X\) is called \emph{symplectic} if its pullback acts trivially on \(\sigma\). The groups of automorphisms and birational transformations of \(X\) are denoted \(\Aut(X)\) and \(\Bir(X)\), respectively.

Let \(X\) be an irreducible holomorphic symplectic manifolds of type \(\OG6\). A \emph{marking} is the choice of an isometry \(\eta\colon \HH^2(X,\IZ) \rightarrow \bL \coloneqq 3\bU \oplus 2[-2]\). Given a \emph{marked pair} \((X,\eta)\) of type \(\OG6\), we consider the associated \emph{representation map}
\[
    \eta_*\colon \Bir(X) \rightarrow \OO(\bL), \quad f \mapsto \eta \circ (f^{-1})^*\circ \eta^{-1}.
\]

We say that an isometry \(g \in \OO(\bL)\) is \emph{induced} by an automorphism (resp. birational transformation) if there exists a marked pair \((X,\eta)\) of type \(\OG6\) and an element \(f \in \Aut(X)\) (resp. \(f \in \Bir(X)\)) such that \(\eta_*(f) = g\).

We define
\begin{align*}
    \cW^\pex_{\OG6} & \coloneqq \Set{w \in \bL}{w^2 = -2, (w, \bL) = 2} \cup \Set{w \in \bL}{w^2 = -4, (w, \bL) = 2} \\
    \cW_{\OG6} & \coloneqq \cW_{{\OG6}}^{\pex} \cup \Set{w \in \bL}{w^2 = -2, (w, \bL) = 1}.
\end{align*}
Given a signed Hodge structure on \(\bL\), a chamber defined by \(\cW^\pex_{\OG6}\) is called an \emph{exceptional} chamber, whereas a chamber defined by \(\cW_{\OG6}\) is called a \emph{Kähler} chamber.


The Torelli theorem for irreducible holomorphic symplectic manifolds was first proven by Verbitsky~\cite{Verbitsky:Torelli,Verbitsky:Torelli.errata} (cf. also Huybrechts~\cite{Huybrechts:basic.results,Huybrechts:global.Torelli} and Markman~\cite{Markman:survey}). 
We state here a version which is convenient for our purposes.


\begin{theorem}[Torelli theorem for manifolds of type \(\OG6\)] \label{thm:torelli.OG6}
A subgroup \(G \subset \OO(\bL)\) is induced by a group of symplectic  automorphisms (resp. symplectic birational transformations) on a manifold of type~\(\OG6\) if and only if there exist a signed Hodge structure on \(\bL\) and a Kähler (resp. exceptional) chamber \(C\) such that
\[
\pushQED{\qed} 
    G \subset \OO^+_\symplectic(\bL, C). \qedhere
\popQED
\]
\end{theorem}
\proof
Suppose that \(G\) is induced by a group of symplectic automorphisms on a manifold \(X\) of type \(\OG6\). It follows that \(G\) is contained in the monodromy group of \(X\), which is \(\OO^+(\bL)\) by a result of Mongardi and Rapagnetta~\cite[Theorem~5.5]{Mongardi.Rapagnetta:OG6.monodromy.birational.geometry}. 
The Hodge decomposition and the canonical choice of the positive cone induce a signed Hodge structure on \(\bL\). All automorphisms of \(X\) preserve the Kähler cone, which is a chamber~\(C\) cut out by vectors in \(\cW_{\OG6}\) by \cite[Theorem~1.2]{Mongardi.Rapagnetta:OG6.monodromy.birational.geometry}, i.e. a Kähler chamber. Therefore, \(G \subset \OO^+_\symplectic(\bL, C)\).

Conversely, suppose there exist a signed Hodge structure on \(\bL\) and a Kähler chamber \(C\) such that \(G \subset \OO^+_\symplectic(\bL, C)\).
Then, by Huybrechts' theorem on the surjectivity of the period map~\cite[Theorem 8.1]{Huybrechts:basic.results} there exists a manifold \(X\) of type \(\OG6\) whose Hodge decomposition induces the given signed Hodge structure. Moreover, we can assume that \(C\) is the Kähler cone of \(X\). As \(G\) consists of monodromy transformations and respects the Kähler cone, it is induced by automorphisms of \(X\) by Markman's Hodge theoretic version of the Torelli theorem \cite[Theorem 1.3]{Markman:survey}. Obviously, the automorphisms must be symplectic.

The same proof applies for birational transformations considering the fundamental  exceptional chamber of \(X\) (\cite[Definition~5.2]{Markman:survey}), which is the distinguished exceptional chamber containing a K\"ahler class, instead of the Kähler cone, by \cite[Corollary~5.7]{Markman:survey}.
\endproof

\begin{theorem} \label{thm:automorphisms}
A finite subgroup \(G \subset \OO(\bL)\) is induced by a group of symplectic automorphisms on a manifold of type \(\OG6\) if and only if \(\bL_G\) is negative definite and 
\begin{equation} \label{eq:thm:automorphisms}
    \bL_G \cap \cW_{\OG6} = \emptyset.
\end{equation}
\end{theorem}
\proof
The two conditions are necessary by \autoref{lem:coinvariant.negative.definite} and \autoref{lem:no.walls.in.coinv}. 

Conversely, assume that they hold.
By \autoref{lem:L^1,1=L_G} there exists a signed Hodge structure on \(\bL\) such that \(\bL^{1,1} \cap \bL = \bL_G\) and \(G\subset \OO^+_\symplectic(\bL)\). Since \(\bL^{1,1} \cap \cW_{\OG6} = \bL_G \cap \cW_{\OG6} = \emptyset\), the decomposition \eqref{eq:exc.chamb} is trivial: there is only one Kähler chamber, namely \(\sP\) itself. 
Therefore, \(\OO^+_\symplectic(\bL) = \OO^+_\symplectic(\bL, \sP)\) and we conclude by \autoref{thm:torelli.OG6}.
\endproof 

\begin{theorem} \label{thm:bir.transformations}
A finite subgroup \(G \subset \OO(\bL)\) is induced by a group of symplectic birational transformations on a manifold of type \(\OG6\) if and only if \(\bL_G\) is negative definite and 
\begin{equation} \label{eq:thm:bir.transformations}
    \bL_G \cap \cW^\pex_{\OG6} = \emptyset.
\end{equation}
\end{theorem}
\proof An analogous argument as in \autoref{thm:automorphisms} applies: there exists a signed Hodge structure on \(\bL\) such that \(\sP\) itself is the only exceptional chamber and \(G \subset \OO^+_\symplectic(\bL) = \OO^+_\symplectic(\bL, \sP)\).
\endproof

\section{Proofs} \label{sec:sp.OG6}

Recall that \(\bL \coloneqq 3\bU \oplus 2[-2]\). We set \(\bLambda \coloneqq 5\bU\), \(\bR \coloneqq 2[2]\) and we also fix a primitive embedding \(\bL \hookrightarrow \bLambda\), necessarily with \(\bL^\perp \cong \bR\).
From now on, \(g \in \OO(\bL)\) will denote an isometry of order~\(m\). 

We first consider the case in which \(m\) is a prime number in~\autoref{subsec:prime.order}. This is enough to prove \autoref{thm:aut.trivial.action} in~\autoref{proof:thm:aut.trivial.action}.
We then go on to classify the cases in which \(m\) is a power of \(2\) in~\autoref{subsec:powers.2}, which we need for the proof of \autoref{thm:bir.trivial.action} in~\autoref{proof:thm:bir.trivial.action}. 
Finally, we treat the remaining cases in \autoref{subsec:other.orders}, and we finish off the proof of \autoref{thm:main.sp.OG6} in~\autoref{proof:thm:main.sp.OG6}.

\subsection{Prime order} \label{subsec:prime.order}
We start with two preliminary observations.

\begin{lemma} \label{lem:p=2,3,5}
If \(g \in \OO(\bL)\) has prime order \(p\) and \(\bL_g\) is negative definite, then \(p \in \{2,3,5\}\).
\end{lemma}
\proof
Since \(\sign(\bL) = (3,5)\), we have \(\rank(\bL_g) \leq 5\), so the statement follows from \autoref{lem:prime.ord}. 
\endproof

\begin{lemma} \label{lem:|O(L^sharp)|=2}
The group \(\OO(\bL^\sharp)\) has order \(2\).
\end{lemma}
\proof Indeed, \(\OO(\bL^\sharp)\) is generated by the involution exchanging the two generators of \(\bL^\sharp\) whose square is equal to \(3/2 \mod 2\IZ\).
\endproof 

For a lattice \(N\) we introduce the set
\[
    \cC(N) \coloneqq \Set{v \in N}{v^2 = -2 \text{ or } v^2 = -4}.
\]
Note that if \(N\) is negative definite, the set \(\cC(N)\) is finite and can be explicitly determined, for instance using the \texttt{sageMath} function \path{short_vector_list_up_to_length} found in \path{sage.quadratic_forms.quadratic_form} \cite{sage}.
The following technical lemma turns out to be of the utmost importance in dealing with condition~\eqref{eq:thm:bir.transformations} of \autoref{thm:bir.transformations}.

\begin{lemma} \label{lem:main.divisibility.lemma}
Let \(N\) be a lattice such that for every \(\alpha \in N^\sharp\) with 
\[
    \alpha^2 \in \set{3/2, 1} \mod 2\IZ
\]
there exists \(v \in \cC(N)\) with \(2 \mid (v, N)\) such that the class of \(v/2\) in \(N^\sharp\) is \(\alpha\). Then, for every primitive embedding \(N \hookrightarrow \bL\) it holds 
\[
    N \cap \cW_{\OG6}^\pex = \emptyset \quad \text{ if and only if } \quad |{\det(N^\perp)}| = |{\det(\bL)} \cdot {\det(N)}|.
\]
\end{lemma}
\proof 
Let \(H \subset N^\sharp\) be the gluing subgroup of a primitive embedding \(N \hookrightarrow \bL\) and \(M = N^\perp\). The condition \(|{\det(M)}| = |{\det(\bL)} \cdot {\det(N)}|\) is equivalent by \eqref{eq:gluing.embedding.index} to \(H = N^\sharp\), which in turn is equivalent to \(H^\perp = 0\).

If this holds, each vector \(v \in N\) satisfies \((v, \bL) = 1\) by \autoref{cor:L.div.v.easy.case}, so \(N \cap \cW_{\OG6}^\pex = \emptyset\).

Conversely, take a nonzero element \(\alpha \in H^\perp\). 
Then, thanks to \eqref{eq:L^sharp=Gamma^perp/Gamma}, we can identify the element \((0,\alpha) \in M^\sharp \oplus N^\sharp\) with a nonzero element \(\beta \in \bL^\sharp\). 
Therefore,
\[
    \alpha^2 = (0,\alpha)^2 = \beta^2 \in \set{3/2, 1} \mod 2\IZ,
\]
so by hypothesis there exists \(v \in \cC(N)\) such that \(v/2 = \alpha \in N^\sharp\).
Then, \((v, \bL) = 2\) by \autoref{lem:L.div.v}, hence \(v \in N \cap \cW_{\OG6}^\pex\).
\endproof 

\begin{remark}
Notably for our purposes, the lattices \([-2],[-4],\bA_3,\bD_4\) and \(\bD_5\) satisfy the hypothesis of \autoref{lem:main.divisibility.lemma}, as one can readily check by inspection.
\end{remark}

\begin{proposition} \label{prop:order.2}
If an isometry \(g \in \OO(\bL)\) of order \(2\) is induced by a symplectic birational transformation, then \(\bL_g\) is isomorphic to either \([-2], 2[-2], 3[-2]\) or \(\bD_4\), and condition~\eqref{eq:cond.det} holds.
\end{proposition}
\proof 
Let \(N = \bL_g\) and suppose first that \(g^\sharp = \id\).

By \autoref{lem:coinv.m-elementary}, \(N\) is an even \(2\)-elementary negative definite lattice of rank \(\leq 5\). Applying \cite[Theorem 3.6.2]{Nikulin:int.sym.bilinear.forms}, we see that \(N\) is isometric to one of the following lattices:
\[ 
    [-2], 2[-2], 3[-2], \bD_4, 4[-2], \bD_4 \oplus [-2], 5[-2].
\]
By inspection, we see that all these lattices satisfy the hypothesis of \autoref{lem:main.divisibility.lemma}. Therefore, condition~\eqref{eq:thm:bir.transformations} of \autoref{thm:bir.transformations} implies condition~\eqref{eq:cond.det}.
This excludes the cases in which \(N \cong 4[-2], \bD_4 \oplus [-2], 5[-2]\), since they do not admit such a primitive embedding.

Now we turn to the case \(g^\sharp \neq \id\).

Consider the embedding \(\bL \hookrightarrow \bLambda\). We can extend the involution \(g\) to an isometry \(f \in \OO(\bLambda)\), also with \(|f| = 2\), by the involution \(h \in \OO(\bR)\) exchanging the generators on \(\bR\).
Since \(\bLambda_f\) is an even \(2\)-elementary hyperbolic lattice of rank \(\leq 6\), \autoref{thm:involutions.unimodular} yields the following possibilities for its isometry class: 
\begin{gather*}
    [2], \bU, \bU(2), [2]\oplus[-2], \bU \oplus [-2], \bU(2) \oplus [-2], \bU \oplus 2[-2], \bU(2) \oplus 2[-2], \\
    \bU \oplus 3[-2], \bU(2) \oplus 3[-2], \bU \oplus \bD_4, \bU \oplus 4[-2].
\end{gather*}

We then compute the possible embeddings \(\bR_h \cong [4] \hookrightarrow \bLambda_f\) to find \(N \cong (\bR_h)^\perp\).
The result of this computation is that \(N\) is isometric to one of the following lattices:
\begin{gather*} 
    [-4], [-2] \oplus [-4], \bA_3, 2[-2] \oplus [-4], \bA_3 \oplus [-2], 3[-2] \oplus [-4], \\ \bD_5, \bD_4 \oplus [-4], \bA_3 \oplus 2[-2], 4[-2] \oplus [-4].
\end{gather*}

By inspection, we see again that all these lattices satisfy the hypothesis of \autoref{lem:main.divisibility.lemma}. 
Therefore, condition~\eqref{eq:thm:bir.transformations} of \autoref{thm:bir.transformations} implies condition~\eqref{eq:cond.det}, 
which in turn implies implies that \(g^\sharp = \id\) by \autoref{lem:cond.det}, contradicting the assumption. \endproof 

In the case of odd prime order \(p = |g| > 2\), we have \(g^\sharp = \id\). Hence, we can use the identity on \(\bR\) to extend \(g\) to an isometry \(f \in \OO(\bLambda)\). 
By \cite[Theorem~1.1]{Brandhorst.Cattaneo:prime.order.unimodular}, we can classify the possible sublattices \(\bL_g \cong \bLambda_{f}\), which are \(p\)-elementary (see also \cite[Chapter~15, Theorem~13]{Conway.Sloane:sphere.packings.lattices.groups}). 
Since \(\det(\bL_g)\) is odd, the embedding subgroup of \(\bL_g \hookrightarrow \bL\) is trivial. Thus, condition~\eqref{eq:cond.det} holds automatically and we obtain the following results.

\begin{proposition} \label{prop:order.3}
If an isometry \(g \in \OO(\bL)\) of order~\(3\) is induced by a symplectic birational transformation, then \(\bL_g\) is isomorphic to either \(\bA_2\) or \(2\bA_2\),
and condition \eqref{eq:cond.det} holds. \qed
\end{proposition}

\begin{proposition} \label{prop:order.5}
If an isometry \(g \in \OO(\bL)\) of order~\(5\) is induced by a symplectic birational transformation, then \(\bL_g\) is isomorphic to \(\bA_4\),
and condition \eqref{eq:cond.det} holds. \qed
\end{proposition}

\subsection{Proof of \autoref{thm:aut.trivial.action}} \label{proof:thm:aut.trivial.action}

Let us assume that a nontrivial isometry \(g \in \OO(\bL)\) of finite order is induced by a symplectic automorphism. 
Up to passing to a power of \(g\), we can suppose without loss of generality that \(p = |g|\) is a prime number. 

By \autoref{thm:automorphisms} the coinvariant sublattice \(\bL_g\) is negative definite, hence \autoref{lem:p=2,3,5} applies. 
A group of automorphisms is in particular a group of birational transformations. Hence, by \autoref{prop:order.2}, \ref{prop:order.3} and \ref{prop:order.5}, we know that \(\bL_g\) must contain a vector \(w\) of square \(-2\).

Since \((w, \bL) \mid w^2\), either \((w, \bL) = 1\) or \((w, \bL) = 2\). In both cases, \(w \in \cW_{\OG6}\).
Therefore, \(\bL_g \cap \cW_{\OG6} \neq \emptyset\), which contradicts the second condition of \autoref{thm:automorphisms}. \qed

\subsection{Powers of \texorpdfstring{\(2\)}{2}} \label{subsec:powers.2}
We now examine the cases in which \(|g| = 2^k\) for \(k > 1\).

\begin{lemma} \label{lem:old.prop:C}
If \(g \in \OO(\bL)\) has order~\(m\) with \(4 \mid m\), then \(\bL_g\) is an \(m\)-elementary lattice.
\end{lemma}
\proof 
If \(g^\sharp = \id\), \autoref{lem:coinv.m-elementary} applies immediately. If \(g^\sharp \neq \id\), we can argue as follows.

We consider again the embedding \(\bL \hookrightarrow \bLambda\) and the extension of \(g\) to an isometry \(f \in \OO(\bLambda)\) by the involution \(h \in \OO(\bR)\) exchanging the generators of \(\bR\).
By \autoref{lem:unimodular.m-el} the lattice \(\bLambda_f\) is \(m\)-elementary. Hence, the claim holds by \autoref{lem:embedding.m-elementary}, because \(\bR_h \cong [4]\) and \(\bL_g \cong \bR_h^\perp \subset \bLambda_f\).
\endproof 

\begin{proposition} \label{prop:order.4}
If an isometry \(g \in \OO(\bL)\) of order~\(4\) is induced by a symplectic birational transformation, then \(\bL_g\) is isomorphic to either \(\bA_3, \bD_4, \bA_3 \oplus [-2]\) or \(\bD_5\), and condition~\eqref{eq:cond.det} holds.
\end{proposition}
\proof 
The lattice \(N \coloneqq \bL_g\) is \(4\)-elementary by \autoref{lem:old.prop:C},
and negative definite by \autoref{thm:bir.transformations}.
We can list all (nonzero) \(4\)-elementary negative definite lattices satisfying \(\rank(N) \leq 5\) and \(|{\det(N)}| \leq 4^5 = 1024\) by \eqref{eq:bound.m-elementary}. 
Since \(N\) admits a primitive embedding \(N \hookrightarrow \bL\), a computation shows that \(N\) must be isomorphic to one of the following lattices:
\begin{gather*}
    [-2], [-4], 2[-2], [-2]\oplus[-4], 2[-4], \bA_3, 3[-2], 2[-2]\oplus[-4], [-2]\oplus 2[-4], 3[-4],
    \bD_4, \\ \bA_3\oplus[-2], \bA_3 \oplus [-4], 4[-2], 3[-2]\oplus[-4], 2[-2]\oplus 2[-4], [-2]\oplus 3[-4], 
    \bD_5, \bD_4 \oplus [-2], \\ \bD_4\oplus[-4],\bA_3 \oplus 2[-2], \bA_3 \oplus [-2] \oplus [-4], 5[-2],
    3[-2]\oplus 2[-4], 2[-2] \oplus 3[-4].
\end{gather*}

For all these lattices, the hypothesis of \autoref{lem:main.divisibility.lemma} is satisfied. Therefore, condition~\eqref{eq:thm:bir.transformations} forces condition~\eqref{eq:cond.det}. The only lattices admitting such a primitive embedding are the following ones:
\begin{gather*}
    [-2], [-4], 2[-2], [-2]\oplus[-4], 2[-4], \bA_3, 3[-2], 2[-2]\oplus[-4], [-2]\oplus 2[-4], 3[-4], \\ \bD_4, \bA_3\oplus[-2], \bD_5.
\end{gather*}

Finally, we observe that the restriction \(g' = g|_N \in \OO(N)\) satisfies \(|g'| = 4\), \(N^{g'} = 0\) and \((g')^\sharp = \id\) by \autoref{lem:cond.det}. A direct computation shows that only the lattices listed in the statement admit such an isometry \(g' \in \OO(N)\).
\endproof

\begin{corollary} \label{prop:A3.if.m|4}
If an isometry \(g \in \OO(\bL)\) of order divisible by~\(4\) is induced by a symplectic birational transformation, then there exists a primitive embedding \(\bA_3 \hookrightarrow \bL_g\).
\end{corollary}
\proof 
Let \(g'\) be a power of \(g\) with \(|g'| = 4\). By inspection of all cases in~\autoref{prop:order.4}, we see that \(\bA_3 \hookrightarrow \bL_{g'}\). Therefore, \(\bL_g\) contains a sublattice isomorphic to \(\bA_3\). Such an embedding is necessarily primitive, as \(\bA_3\) does not admit even overlattices of index \(>1\), so we conclude. 
\endproof 

\begin{lemma} \label{lem:IC.lemma}
A lattice \(N\) admitting primitive embeddings \(\bA_n \hookrightarrow N \hookrightarrow \bL\) satisfies
\[
    \ell_p(N^\sharp) \leq \begin{cases} \rank(N) - n & \text{if \(p \nmid n+1\)} \\ \rank(N) - n + 1 & \text{if \(p \mid n+1\)}. \end{cases}
\]
Moreover, if \(p \neq 2\), then
\[
    \ell_p(N^\sharp) \leq 8 - \rank(N).
\]
\end{lemma}
\proof
Let \(B \coloneqq \bA_n^\perp \subset N\).
From the identification~\eqref{eq:L^sharp=Gamma^perp/Gamma} for the embedding \(\bA_n \hookrightarrow N\) and the fact that \(\bA_n^\sharp \cong \IZ/(n+1)\IZ\)
we see that
\[
    \ell_p(N^\sharp) \leq \ell_p(\bA_n^\sharp) + \ell_p(B^\sharp) \leq \ell_p(\bA_n^\sharp) + \rank(B) \leq \begin{cases} 0 + (\rank(N) - n) & \text{if \(p \nmid n+1\)} \\ 1 + (\rank(N) - n) & \text{if \(p \mid n+1\)}. \end{cases}
\]

Suppose now that \(p \neq 2\) and let \(M = N^\perp \subset \bL\). From the identification~\eqref{eq:(S^perp)^sharp=Xi^perp/Xi} for the embedding \(N \hookrightarrow \bL\) we obtain
\[
    \ell_p(N^\sharp) \leq \ell_p(\bL^\sharp) + \ell_p(M^\sharp) \leq 0 + \rank(M) = 8 - \rank(N). \qedhere 
\]
\endproof 

\begin{proposition} \label{prop:order.8}
If an isometry \(g \in \OO(\bL)\) of order~\(8\) is induced by a symplectic birational transformation, then \(\bL_g\) is isomorphic to \(\bD_5\), and condition~\eqref{eq:cond.det} holds.
\end{proposition}
\proof
Let \(N \coloneqq \bL_g\). By inspection of all cases in~\autoref{prop:order.4}, we see that there exists a primitive embedding \(\bA_3 \hookrightarrow \bL_{g^2}\). 
Hence, thanks to \autoref{thm:bir.transformations}, \autoref{lem:old.prop:C}, and the existence of \(\bL_{g^2} \hookrightarrow N\), we infer that \(N\) is an \(8\)-elementary negative definite lattice admitting primitive embeddings
\(\bA_3 \hookrightarrow N \hookrightarrow \bL\). In particular, \(\rank(N) \leq 5\). By \autoref{lem:IC.lemma} we have \(\ell_2(N^\sharp) \leq 3\), so \(|{\det(N)}| \leq 8^3 = 512.\)

A computation with the algorithm explained in \autoref{subsec:m-elementary} shows that \(N\) must be isomorphic to either \(\bD_4, \bD_5\) or \(\bD_4 \oplus [-2]\).
By \autoref{lem:main.divisibility.lemma} and condition~\eqref{eq:thm:bir.transformations}, we can exclude the case \(N \cong \bD_4 \oplus [-2]\). 
Finally, the restriction \(g' = g|_N\) satisfies \(|g'| = 8\), \(N^{g'} = 0\) and \((g')^\sharp = \id\). There is no such isometry \(g' \in \OO(\bD_4)\), hence only \(N \cong \bD_5\) is possible.
\endproof 

\begin{proposition} \label{prop:order.16}
An isometry \(g \in \OO(\bL)\) induced by a symplectic birational transformation cannot have order~\(16\).
\end{proposition}
\proof
Suppose such an isometry \(g\) exists.
Then, \(\bL_{g^2} \cong \bD_5\) by \autoref{prop:order.8}. 
Now, \(\bL_{g^2}\) embeds primitively into \(\bL_g\) and \(\bL_g\) is negative definite. Therefore,
\[
    5 = \rank(\bL_{g^2}) \leq \rank(\bL_g) \leq 5,
\] 
which forces \(\rank(\bL_g) = 5\) and \(\bL_g \cong \bL_{g^2} \cong \bD_5\). But \(\bD_5\) admits no isometry of order \(16\), hence such \(g\) cannot exist.
\endproof 

\subsection{Proof of \autoref{thm:bir.trivial.action}} \label{proof:thm:bir.trivial.action}
Let \(g \in \OO(\bL)\) be induced by a symplectic birational transformation and suppose that \(g^\sharp \neq \id \).
Thanks to \autoref{lem:|O(L^sharp)|=2}, up to passing to an odd power of \(g\) we can suppose without loss of generality that \(|g| = 2^k\) for some \(k \in \IN\).

From \autoref{prop:order.16} it follows that \(k \leq 3\). Therefore, condition~\eqref{eq:cond.det} holds because of \autoref{prop:order.2}, \ref{prop:order.4} or \ref{prop:order.8}. Consequently, it must be \(g^\sharp = \id\) by \autoref{lem:cond.det}, contradicting the assumption. \qed

\subsection{Other orders} \label{subsec:other.orders}
\autoref{thm:bir.trivial.action} and \autoref{lem:coinv.m-elementary} imply together that \(\bL_g\) is \(|g|\)-elementary if \(g\) is induced by a symplectic birational transformation, a fact which we will not repeat in the rest of the paper.
In the next proof, the following lattice appears:
\[
 \btA = {\small \begin{pmatrix} -2 & 1 & -1 & -1 \\
1 & -2 & 1 & 1 \\
-1 & 1 & -2 & 0 \\
-1 & 1 & 0 & -4 \end{pmatrix}}.
\]

\begin{proposition} \label{prop:order.6}
If an isometry \(g \in \OO(\bL)\) of order~\(6\) is induced by a symplectic birational transformation, 
then \(\bL_g\) is isomorphic to one of the following lattices:
\(
    \bA_2 \oplus [-2], \bD_4, \bA_2 \oplus 2[-2], 2\bA_2 \oplus [-2].
\)
Moreover, condition~\eqref{eq:cond.det} holds.
\end{proposition}
\proof
Let \(N \coloneqq \bL_g\). By~\autoref{prop:order.3}, it holds \(\bL_{g^2} \cong \bA_2\) or \(\bL_{g^2} \cong 2\bA_2\). Since \(\bL_{g^2} \hookrightarrow \bL_g\), in both cases there must exist a (necessarily primitive) embedding \(\bA_2 \hookrightarrow N\).
Hence, it holds \(\ell_2(N^\sharp) \leq 3\) and \(\ell_3(N^\sharp) \leq 3\) by \autoref{lem:IC.lemma} and the fact that \(\rank(N) \leq 5\) because \(N\) is negative definite; consequently, \(|{\det(N)}| \leq 6^3 = 216\). A computation shows that \(N\) must be isomorphic to one of the following lattices:
\begin{gather*}
    \bA_2, \bA_2 \oplus [-2], \bA_2 \oplus [-6], \bD_4, 2\bA_2, \btA, \bA_2 \oplus 2[-2], \bA_2 \oplus \bA_2(2), \bA_2 \oplus [-2] \oplus [-6], \\ \bA_2 \oplus 2[-6],
     \bD_5, \bD_4 \oplus [-2], 2\bA_2 \oplus [-2], \btA \oplus [-2], \bA_2 \oplus 3[-2], \bD_4 \oplus [-6], 2\bA_2 \oplus [-6], \\
     \bA_2 \oplus 2[-2] \oplus [-6], \btA \oplus [-6],
     \bA_2 \oplus \bA_2(2) \oplus [-2],
     \bA_2 \oplus [-2] \oplus [-6], \bA_2 \oplus \bA_2(2) \oplus [-6].
\end{gather*}

Since \(3\) does not divide \(\det(\bL)\), all elements of order \(3\) in \(N^\sharp\) must belong to the gluing subgroup \(H\) of \(N \hookrightarrow \bL\). Therefore, the restriction \(g' = g|_N \in \OO(N)\) satisfies \(|g'| = 6\), \(N^{g'} = 0\), and induces the identity on the elements of order \(3\) in \(N^\sharp\).
By inspection we see that only the following lattices admit such an isometry~\(h\):
\[
\bA_2 \oplus [-2], \bD_4, \bA_2 \oplus 2[-2], \bD_5, \bD_4 \oplus [-2], \bA_2 \oplus [-2].
\]

\autoref{lem:main.divisibility.lemma} implies that condition \eqref{eq:cond.det} must hold in order for \(\cW^\pex_{\OG6} \cap N = \emptyset\) to be true, hence \((g')^\sharp = \id\) because of \autoref{lem:cond.det}. This excludes \(\bD_5\) and \(\bD_4 \oplus [-2]\).
\endproof

\begin{proposition} \label{prop:order.9}
An isometry \(g \in \OO(\bL)\) induced by a symplectic birational transformation cannot have order~\(9\).
\end{proposition}
\proof 
By looking at \(g^3\) and \autoref{prop:order.3}, we see that \(N \coloneqq \bL_g\) is a \(9\)-elementary negative definite lattice admitting primitive embeddings \(\bA_2 \hookrightarrow N \hookrightarrow \bL\).
Hence, it holds \(\ell_3(N^\sharp) \leq 3\) by \autoref{lem:IC.lemma}; consequently, \(|{\det(N)}| \leq 9^3 = 729\).
We compute that \(N\) can be either \(\bA_2, 2\bA_2\) or \(\bA_2 \oplus \bA_2(3)\), but none of these lattices admits an isometry of order~\(9\).
\endproof 

\begin{proposition} \label{prop:order.10}
If an isometry \(g \in \OO(\bL)\) of order~\(10\) is induced by a symplectic birational transformation, then \(\bL_g\) is isomorphic to \(\bA_4 \oplus [-2]\), and condition~\eqref{eq:cond.det} holds.
\end{proposition}
\proof 
Arguing as above, \(N \coloneqq \bL_g\) must be a \(10\)-elementary negative definite lattice admitting primitive embeddings \(\bA_4 \hookrightarrow N \hookrightarrow L\). 
Thus, if \(\rank(N) = 4\) then \(N \cong \bA_4\), but this lattice does not have any isometry \(g \in \OO(N)\) of order~\(10\) acting trivially on the discriminant group such that \(N^{h} = 0\). 

Suppose that \(\rank(N) = 5\). By \autoref{lem:IC.lemma} it holds \(\ell_2(N^\sharp) \leq 1\) and \(\ell_5(N^\sharp) \leq 2\); consequently,  \(|{\det(N)}| \leq 2 \cdot 5^2 = 50\). 
We compute that \(N \cong \bA_4 \oplus [-2]\) or \(N \cong \bA_4 \oplus [-10]\), but the second lattice does not have admit any isometry \(g' \in \OO(N)\) of order \(10\) acting trivially on the elements of \(N^\sharp\) of order \(5\) such that \(N^{g'} = 0\). 
\endproof 

\begin{proposition} \label{prop:order.12}
If an isometry \(g \in \OO(\bL)\) of order~\(12\) is induced by a symplectic birational transformation, then \(\bL_g\) is isomorphic to \(\bD_5\), and condition~\eqref{eq:cond.det} holds.
\end{proposition}
\proof 
Let \(N \coloneqq \bL_g\), which is a \(12\)-elementary negative definite lattice.
Looking at \(g^3\), we see that there exists primitive embeddings \(\bA_3 \hookrightarrow N\) (by inspection of the cases in \autoref{prop:order.4}). 
We infer from \autoref{lem:IC.lemma} that \(\ell_2(N^\sharp) \leq 3\) and \(\ell_3(N^\sharp) \leq 2\), so \(|{\det(N)}| \leq 4^3 \cdot 3^2 = 576\). Therefore, \(N\) is isomorphic to one of the following lattices:
\begin{gather*}
\bA_3 \oplus [-2], \bD_5, \bA_5, \bD_4 \oplus [-2], \bA_3 \oplus \bA_2, \bA_3 \oplus 2[-2], \btA \oplus [-2], \bA_3 \oplus [-2] \oplus [-4], \\ 
\bA_3 \oplus [-2] \oplus [-6], \bA_3 \oplus [-2] \oplus [-12].
\end{gather*}
Since there must exist \(g' \in \OO(N)\) with \(|g'| = 12\), we see that \(N\) is isomorphic to either \(\bD_5\) or \(\bD_4 \oplus [-2]\). But then condition \eqref{eq:cond.det} must hold because of \autoref{lem:main.divisibility.lemma}, so it should also be \((g')^\sharp = \id\). Such an isometry \(g'\) exists only if \(N \cong \bD_5\).
\endproof 

\begin{proposition} \label{prop:order.15.25}
An isometry \(g \in \OO(\bL)\) induced by a symplectic birational transformation cannot have order~\(15\) or \(25\).
\end{proposition}
\proof
Let \(m \in \{15, 25\}\).
The coinvariant sublattice \(N \coloneqq \bL_g\) is an \(m\)-elementary negative definite lattice admitting primitive embeddings \(\bA_4 \hookrightarrow N \hookrightarrow \bL\) by \autoref{prop:order.5}. Since \(m\) is odd, we have necessarily \(\rank(N) = 4\) by \autoref{lem:Trieste}, hence \(N \cong \bA_4\), but there is no isometry in \(\OO(\bA_4)\) of order \(m\).
\endproof 

\begin{proposition} \label{prop:order.20.24}
An isometry \(g \in \OO(\bL)\) induced by a symplectic birational transformation cannot have order~\(20\) or \(24\).
\end{proposition}
\proof 
Suppose \(|g| = 20\). Considering \(g^2\), we see from \autoref{prop:order.10} that it should be \(\bL_g \cong \bA_4 \oplus [-2]\), but this lattice does not admit any isometry of order~\(20\). An analogous argument holds when \(|g| = 24\).
\endproof 

\subsection{Proof of \autoref{thm:main.sp.OG6}} \label{proof:thm:main.sp.OG6}
Suppose that an isometry \(g \in \OO(\bL)\) of finite order~\(m\) is induced by a symplectic birational transformation on a manifold of type \(\OG6\).
By \autoref{thm:bir.transformations} the coinvariant sublattice \(\bL_g\) is negative definite, hence by \autoref{lem:p=2,3,5} only \(2,3\) and \(5\) appear as prime factors of~\(m\). 
As the orders \(16, 9\) and \(25\) are not possible by \autoref{prop:order.16}, \ref{prop:order.9} and \ref{prop:order.15.25}, respectively, it holds \(m \mid 2^3 \cdot 3 \cdot 5 = 120\), which means
\[
    m \in \set{1,2,3,4,5,6,8,10,12,15,20,24,30,40,60,120}.
\]
Now, \autoref{prop:order.15.25} excludes \(15,30,60,120\), and \autoref{prop:order.20.24} excludes \(20,24,40\). Therefore, condition \eqref{eq:possible.orders} holds.

From \autoref{prop:order.2}, \ref{prop:order.3}, \ref{prop:order.5}, \ref{prop:order.4}, \ref{prop:order.8}, \ref{prop:order.6}, \ref{prop:order.10} and \ref{prop:order.12}, we know the isomorphism class of~\(\bL_g\). 
Moreover, equation \eqref{eq:cond.det} always holds, which means that the embedding subgroup of \(\bL_g \hookrightarrow \bL\) is trivial.
This enables us to reconstruct the genus of~\(\bL^g\), hence \(\bL^g\) itself, being always unique in its genus.
As a result of our computations, the pair \((\bL^g,\bL_g)\) appears in \autoref{tab:sp.OG6}.

Conversely, one can check that the column `example' provides isometries with the given invariant and coinvariant lattices. By \autoref{cor:L.div.v.easy.case}, condition~\eqref{eq:thm:bir.transformations} of \autoref{thm:bir.transformations} holds automatically. Therefore, all isometries contained in \autoref{tab:sp.OG6} are indeed induced by a symplectic birational transformation on some marked pair of type \(\OG6\).
This concludes the proof of the theorem.
\qed

\bibliographystyle{amsplain}
\bibliography{references}

\end{document}